    \documentclass[11pt]{amsart}
 \usepackage{amssymb}
 \theoremstyle{plain}
 \newtheorem{theorem}{Theorem}
 \newtheorem{corollary}{Corollary}
 
  \newtheorem{lemma}{Lemma}
 \newtheorem{proposition}{Proposition}
 \theoremstyle{definition}
 
 \theoremstyle{remark}
 
 \numberwithin{equation}{section}

\setbox0=\hbox{$+$}
\newdimen\plusheight
\plusheight=\ht0\def\+{\;\lower\plusheight\hbox{$+$}\;}

\setbox0=\hbox{$-$}
\newdimen\minusheight
\minusheight=\ht0\def\-{\;\lower\minusheight\hbox{$-$}\;}

\setbox0=\hbox{$\cdots$}
\newdimen\cdotsheight
\cdotsheight=\plusheight
\def\cds{\lower\cdotsheight\hbox{$\cdots$}}

\begin{document}
\title[Convergence Behavior  of  $q$-Continued Fraction on the Unit Circle]
{The   Convergence Behavior of  $q$-Continued Fractions on the Unit
Circle}
\author{Douglas Bowman}
\address{Department Of Mathematical Sciences,
Northern Illinois University,
 De Kalb, IL 60115}
\email{bowman@math.niu.edu }
\author{James Mc Laughlin}
\address{Mathematics Department,
       Trinity College,
       300 Summit Street,
       Hartford, CT 06106-3100}
\email{james.mclaughlin@trincoll.edu}

\thanks{The second author's research supported in part by a
Trjitzinsky Fellowship.}
\date{April,   2,  2002}

\begin{abstract}
In a previous paper,  we showed the existence of an uncountable
set of points on the unit  circle at which the
 Rogers-Ramanujan continued fraction does not converge to a finite value.

In this present paper,  we generalise this result to a
wider class of $q$-continued fractions,  a class which includes
the Rogers-Ramanujan continued fraction and the three
Ramanujan-Selberg continued fractions. We show,  for each
$q$-continued fraction,  $G(q)$,  in this class,  that there is an
uncountable set of points,  $Y_{G}$,  on the unit circle such that
if $y \in Y_{G}$ then $G(y)$ does not converge to a finite value.

We discuss the implications
of our theorems for the convergence of other $q$-continued
fractions,  for example the G\"{o}llnitz-Gordon continued
fraction,  on the unit circle.
\end{abstract}
\maketitle
 \section{Introduction}
In a previous paper \cite{BML01} we studied the convergence
behavior of the celebrated Rogers--Ramanujan
 continued fraction, $R(q)$, which is defined
 for $|q|<1$ by
\begin{equation}\label{rreq}
 R(q):= \frac{q^{1/5}}{1}
\+
 \frac{q}{1}
\+
 \frac{q^{2}}{1}
\+
 \frac{q^{3}}{1}
\+\,\cds.
\end{equation}
Put $K(q)=q^{1/5}/R(q)$.

Worpitzky's Theorem (see \cite{LW92}, pp. 35--36) gives that
$R(q)$ converges to a value in $\hat{\mathbb{C}}$
for any $q$ inside the
unit circle. Here $\hat{\mathbb{C}}$ denotes the extended complex plane.
\begin{theorem}(Worpitzky) Let the continued fraction
$K_{n=1}^{\infty}a_{n}/1$
be such that $|a_{n}|\leq 1/4$ for $n \geq 1$. Then
$K_{n=1}^{\infty}a_{n}/1$ converges. All approximants of
the continued fraction lie in the disc $|w|<1/2$ and the value
of the continued fraction is in the disk $|w|\leq1/2$.
\end{theorem}

Outside the unit circle the odd and even parts of $K(q)$ tend to different
limits.
Suppose $|q|>1$.
For $n \geq 1$, define
\[
K_{n}(q):=1+
\frac{q}{1}
\+
 \frac{q^{2}}{1}
\+
 \frac{q^{3}}{1}
\+ \cds \+
 \frac{q^{n}}{1}.
\]
Then
{\allowdisplaybreaks
\begin{align*}
\lim_{j \to \infty}K_{2 j + 1}(q) &= \frac{1}{K(-1/q)},\\
&\phantom{ass}\\
\lim_{j \to \infty}K_{2 j}(q) &= \frac{K(1/q^{4}) }{q}.
\end{align*}
}
This was stated by Ramanujan without proof and proved by Andrews, Berndt,
Jacobson and Lamphere in 1992 (\cite{ABJL92}).

On the unit circle, there are two cases to consider. The easier of these is where $q$
 is a root of unity, in which case $K(q)$ is periodic.
Schur showed in \cite{S17} that if $q$ is a primitive $m$-th
root of unity, where $m \equiv 0$ $(\text{mod}\,\,5)$, then $K(q)$
diverges in the classical sense and if
 $q$ is a primitive $m$-th
root of unity, $m \not \equiv 0 (\text{mod}\,\,5)$, then $K(q)$ converges
and
\begin{equation}\label{E:ScM}
K(q) =
\lambda \, q^{(1 - \lambda \sigma  m )/5} K(\lambda),
\end{equation}
where $\lambda = \left( \frac{m}{5} \right)$ (the Legendre symbol)
and $\sigma$ is the least positive residue of $m \,(\text{mod}\,\, 5)$.
Note that
$K(1)= \phi = (\sqrt{5}+1)/2$, and $K(-1) = 1/ \phi$.

Remark: Schur's result was essentially proved by Ramanujan, probably earlier than
Schur (see \cite{R57}, p.383). However, he made a calculational error (see \cite{Hg97},
p.56).

\textbf{Question:}
 Does the Rogers-Ramanujan  continued fraction converge or diverge,
in either the classical or general sense, at any point on the unit
circle which is not a root of unity? This question had been open
since Schur's 1917 paper until our paper \cite{BML01}. In trying
to prove  convergence in the classical sense to a finite value,
for example, one immediate difficulty is that Schur's theorem
gives that there is a dense set on the unit circle at which $K(q)$
converges and another dense set at which $K(q)$ diverges. This
immediately renders most of the usual convergence/divergence
tests, which rely either on the partial quotients lying in certain
subsets of $\mathbb{C}$ or on the absolute values of the partial
quotients satisfying certain inequalities, useless.

To discuss this
topic we use the following notation. Let the regular continued
fraction expansion of any irrational $t  \in (0, 1)$ be denoted by
$t=[0, e_{1}(t), e_{2}(t),  \cdots]$. Let the $i$-th approximant
of this continued fraction expansion be denoted by
 $c_{i}(t)/d_{i}(t)$. We will sometimes write $e_{i}$ for $e_{i}(t)$,  $c_{i}$ for $c_{i}(t)$ etc,
if there is no danger of ambiguity. Let $ \phi = ( \sqrt{5}+1)/2$.
In    \cite{BML01}, we proved the following theorem.
 \begin{theorem} \label{T:t1}
Let{ \allowdisplaybreaks
 \begin{align} \label{E:seq}
S&= \{t  \in (0, 1): e_{i+1}(t)  \geq   \phi^{d_{i}(t)} \text{
infinitely often} \}. \\ & \phantom{as}  \notag
 \end{align}}
Then $S$ is an uncountable set of measure zero and,  if $t  \in S$
and $y =  \exp (2  \pi i t)$,  then the Rogers-Ramanujan continued
fraction does not converge to a finite value at $y$.
 \end{theorem}
We were also able to give explicit examples of points $y$ on the unit circle at
which $K(y)$ did not converge to a finite value.
 \begin{corollary} \label{C:ex}
Let $t$ be the number with continued fraction expansion equal $[0,
e_{1}, e_{2},   \cdots]$,  where $e_{i}$ is the integer consisting
of a tower of $i$ twos with an $i$ an top.
 \begin{multline*}
t=[0, 2, 2^{ \displaystyle{2^{2}}}, 2^{ \displaystyle{2^{
\displaystyle{2^{3}}}}},  \cdots]= \\
0.484848484848484848484848484848484848484848484848484848484 \\
84848484848484848484849277885083112437522992318812011  \cdots
 \end{multline*}
If $y =  \exp (2  \pi i t)$ then $K(y)$ does not converge to a finite value.
 \end{corollary}
In  \cite{BML01} we were also able to show the existence of an uncountable set of points
$G$ on the unit circle  such that if $y \in G$, then $K(y)$ does not converge generally.

Remark:  In particular, this latter result shows the existence
of an uncountable set of points $G$ on the unit circle at which
the Rogers-Ramanujan continued fraction does not converge in the
classical sense to values in $\hat{\mathbb{C}}$.

  In
\cite{BML01} we also showed directly from the definition that
$K(q)$ converges in the general sense at $5m$-th roots of unity,
for $m$ a positive integer. This is in contrast to Schur's result
which showed that $K(q)$ does not converge in the classical sense
at $5m$-th roots of unity.

Remark: The divergence behavior of K(q) at $5m$-th roots of unity is an example
of \emph{Thiele oscillation} and convergence in the general sense can also be deduced
from this Thiele oscillation. However, when we consider the case where $q$ is a point
on the unit circle which is \emph{not} a root of unity,  the continued fraction is no longer
periodic, or even limit-periodic,
 and general results about these types of continued fraction can no longer
be applied.

In \cite{BML02} we were able to show, for each $q$-continued fraction in a wide class,
that there existed an uncountable set of points on the unit circle at which the continued fraction
does not converge in the general sense. However, the proof of this is rather long and technical.
It is a much simpler task to prove the divergence in the classical sense and
in this paper we generalise Theorem  \ref{T:t1} to a wider class
of $q$-continued fractions,  a class which includes the
Rogers-Ramanujan continued fraction and the three
Ramanujan-Selberg continued fractions (first studied by Ramanujan  \cite{R57})
investigated by Zhang in \cite{Z91}.
{\allowdisplaybreaks
\begin{align}\label{z1}
S_{1}(q):= 1 + \frac{q}{1}
\+
 \frac{q+q^{2}}{1}
\+
 \frac{q^{3}}{1}
\+
 \frac{q^{2}+q^{4}}{1}
\+
\cds ,
\end{align}
}
\begin{align}\label{z2}
S_{2}(q):=
1 +
 \frac{q+q^{2}}{1}
\+
 \frac{q^{4}}{1}
\+
 \frac{q^{3}+q^{6}}{1}
\+
 \frac{q^{8}}{1}
\+
\cds ,
\end{align}
and
\begin{align}\label{z3}
S_{3}(q):=
1 +
 \frac{q+q^{2}}{1}
\+
 \frac{q^{2}+q^{4}}{1}
\+
 \frac{q^{3}+q^{6}}{1}
\+
 \frac{q^{4}+q^{8}}{1}
\+
\cds .
\end{align}
Thus,  we show,  for each of these continued fractions,  the
existence of an uncountable set of points on the unit circle at
which the continued fraction does not converge to  finite values.

 \section{Convergence Behavior of $q$-Continued Fractions on the Unit Circle}

Let
\begin{equation}\label{E:cf1}
G(q):=b_{0}(q) +K_{n=1}^{ \infty}\frac{a_{n}(q) }{b_{n}(q)},
\end{equation}
where  $a_{n}(q)  \in  \mathbb{Z}[q]$,  for  $n \geq 1$, $b_{n}(q)
 \in  \mathbb{Z}[q]$,  for  $n \geq 0$,  and $q$ is a complex variable.

 For the remainder of the paper $P_{n}(q)/$ $Q_{n}(q)$
 denotes the $n$-th approximant of $G(q)$,  $P_{n}/Q_{n}$ if there is no danger of ambiguity.
 It is  well known (see   \cite{LW92},  page 9) that,  for $n  \geq 1$,
 { \allowdisplaybreaks
  \begin{align}
\label{reclema}
 P_{n}Q_{n-1}-P_{n-1}Q_{n} &= (-1)^{n-1} \prod_{i=1}^{n}a_{i}.
  \end{align}
 }
  For the continued fraction defined in  \eqref{E:cf1}
 let
 { \allowdisplaybreaks
  \begin{align} \label{E:fn}
  \chi_{n}(q) :=  \prod_{i=1}^{n}a_{i}(q).
  \end{align}}We prove the following theorem.
  \begin{theorem} \label{T:t3}
 Let $G(q)$ be as in  \eqref{E:cf1}.
 Suppose there exist constants $C_{1}$,  $C_{2} > 0$ and positive integers
$j$ and $d$
 such that if $m  \equiv j (\text{mod }d)$ and $x_{m}$ is a primitive $m$-th root of unity,
 then there exists a positive integer $n(m)$,  where $n(m)  \to  \infty$ as
 $m  \to  \infty$,   such that
 { \allowdisplaybreaks
  \begin{align} \label{E:fnm}
  | \chi_{n(m)}(x_{m})|  \geq C_{1}
  \end{align}
}
 and
 { \allowdisplaybreaks
  \begin{align} \label{E:qn}
  \max  \{ \,  \,  \left| Q_{n(m)}(x_{m})  \right|,  \,  \,
  \left| Q_{n(m)-1}(x_{m})  \right| \,  \,   \} \leq C_{2}.
  \end{align}
}
 Then there is an uncountable set of points $Y_{G}$ on the unit circle such that if
 $y  \in Y_{G}$ then $G(y)$ does not converge to a finite value.
  \end{theorem}
 As a corollary to this theorem,  we will show,  for each of the continued fractions
 $K(q)$,  $S_{1}(q)$,  $S_{2}(q)$  and $S_{3}(q)$,  that there exists an uncountable set of points
 on the unit circle at which the continued fraction does not converge to a finite value.

 Let $G(q)$ be as defined in Equation  \ref{E:cf1} and suppose $G(q)$ converges at
 $q = y$ to $L  \in  \mathbb{C}$,  so that
 $ \lim_{n  \to  \infty}P_{n}(y)/Q_{n}(y) = L$.
 { \allowdisplaybreaks
  \begin{align*}
  \left| \frac{P_{n}(y)}{Q_{n}(y)}- \frac{P_{n-1}(y)}
 {Q_{n-1}(y)} \right|& \leq  \left| \frac{P_{n}(y)}{Q_{n}(y)}-L \right|+ \left| \frac{P_{n-1}(y)}
 {Q_{n-1}(y)}-L \right|. \\
  \end{align*}}
 Thus
 { \allowdisplaybreaks
  \begin{align} \label{E:limeq}
  \lim_{n  \to  \infty} \left| \frac{P_{n}(y)}{Q_{n}(y)}-
  \frac{P_{n-1}(y)}{Q_{n-1}(y)} \right| =0.
  \end{align}
}
 We will exhibit an uncountable set of points for which  \eqref{E:limeq} fails to hold,  so
 that $G(q)$ does not converge to a finite value at any of these points.

 We recall the notation introduced before the statement of
Theorem \ref{T:t1}.  Let the regular continued
fraction expansion of any irrational $t  \in (0, 1)$ be denoted by
$t=[0, e_{1}(t), e_{2}(t),  \cdots]$. Let the $i$-th approximant
of this continued fraction expansion be denoted by
 $c_{i}(t)/d_{i}(t)$. We will sometimes write $e_{i}$ for $e_{i}(t)$,  $c_{i}$ for $c_{i}(t)$ etc,
if there is no danger of ambiguity.
The key idea here is to construct real numbers $t$ in the interval $(0, 1)$,
 each of which has the property that the sequence of approximants to its continued fraction
 expansion contains a subsequence of approximants which are ``sufficiently close'' to $t$ in a
 certain precise sense.
Recall that it is possible to construct a real number $t$ for which  the
 $m$-th approximant,  $c_{m}/d_{m}$,  in its continued fraction expansion,  is as close to
 $t$ as desired by making the $m+1$-st partial quotient,  $e_{m+1}$,  sufficiently large.
 We then set $y:= \exp(2  \pi i t)$.

 The purpose in constructing an irrational number $t$ of this form is to exert a certain
 amount of control over the absolute value of some of the terms in the sequence
 $ \{Q_{n}(y) \}_{n=1}^{ \infty}$.
 If $x_{m}:= \exp(2  \pi i c_{m}/d_{m})$, then $y$ and $x_{m}$ are close enough to keep
 $ \chi_{n(d_{m})}(y)$ close to $ \chi_{n(d_{m})}(x_{m})$ and $Q_{n(d_{m})}(y)Q_{n(d_{m})-1}(y)$
 close to $Q_{n(d_{m})}(x_{m})Q_{n(d_{m})-1}(x_{m})$
 for  infinitely many $d_{i}$ in the sequence $ \{d_{m} \}_{m=1}^{ \infty}$.
 Equations  \eqref{E:fnm} and  \eqref{E:qn} will then give that the sequence
 \[
 \left \{\frac{P_{n}(y)}{Q_{n}(y)}-
\frac{P_{n-1}(y)}{Q_{n-1}(y)} \right  \}_{n=1}^{ \infty}
\]
 contains an
 infinite subsequence which is bounded away from $0$, contradicting the requirement at
\eqref{E:limeq} and giving the result.

 Before proving Theorem  \ref{T:t3} it is necessary to prove some technical lemmas.

  \begin{lemma} \label{L:l1}
 Let $G(q)$ be as in  \eqref{E:cf1}.
 There exist strictly increasing sequences of positive integers
 $ \{ \kappa_{n} \}$,  $ \{ \nu_{n} \}$ and $ \{ \lambda_{n} \}$ such that if $x$ and $y$ are any two
 points on the unit circle then,  for all integers  $n  \geq 0$,
 { \allowdisplaybreaks
  \begin{align} \label{E:qdif}
 &|Q_{n}(x)-Q_{n}(y)|  \leq  \kappa_{n}|x-y|;
  \end{align}}
 { \allowdisplaybreaks
  \begin{align} \label{E:pdif}
 |P_{n}(x)-P_{n}(y)|  \leq  \nu_{n}|x-y|;
  \end{align}  }
 and
 { \allowdisplaybreaks
  \begin{align} \label{E:fdif}
 | \chi_{n}(x)- \chi_{n}(y)|  \leq  \alpha_{n}|x-y|.
  \end{align}  }
  \end{lemma}

  \begin{proof}
 Let $ \{f_{n}(q) \}$ be any sequence of polynomials in $ \mathbb{Z}[q]$.
 Suppose $f_{n}(q) =  \sum_{i=0}^{M_{n}}  \gamma_{i}q^{i}$,  where the $ \gamma_{i}$'s are in
 $ \mathbb{Z}$. Then
 { \allowdisplaybreaks
  \begin{equation*}
 |f_{n}(x)-f_{n}(y)| \leq
   \sum_{i=1}^{M_{n}}| \gamma_{i}|
   \left|x^{i}-y^{i} \right|
 \leq  \sum_{i=1}^{M_{n}}i \, | \gamma_{i}||x-y|.
   \end{equation*}}
  Now set
  $ \delta_{n} =  \max  \left \{ \,  \,   \sum_{i=1}^{M_{n}}i \, | \gamma_{i}|,  \,  \,  \,  1,
    \,  \,  \delta_{n-1}+1 \right \}$.
   Inequality  \eqref{E:qdif} follows by setting $f_{n}(q)=Q_{n}(q)$ and
   $ \delta_{n}= \kappa_{n}$.
   The result for \eqref{E:pdif} and   \eqref{E:fdif} follow similarly.
    \end{proof}

    \begin{lemma} \label{reclem}
   Let $G(q)$ be as in  \eqref{E:cf1} and $ \chi_{n}(q)$as in  \eqref{E:fn}.
   Then,  for $n  \geq 1$,
   { \allowdisplaybreaks
    \begin{align}
    P_{n}(q)Q_{n-1}(q)-P_{n-1}(q)Q_{n}(q)
   &= (-1)^{n-1} \chi_{n}(q).
    \end{align}}
    \end{lemma}
    \begin{proof}
   This follows from  \eqref{reclema}.
    \end{proof}

   Let $ \left \{  \kappa_{n}  \right \}_{n=1}^{ \infty}$ be as in  \eqref{E:qdif},
   $ \left \{  \alpha_{n}  \right \}_{n=1}^{ \infty}$ be as in  \eqref{E:fdif}, and
   $ \{n(m) \}_{m=1}^{ \infty}$,  $j$ and $d$ as in the statement of
Theorem  \ref{T:t3}.
  Let $S_{0}$ be the set of all
   irrational $t  \in (0, 1)$  satisfying
   { \allowdisplaybreaks
    \begin{align}e_{1}(t) & \equiv j \text{ (mod } d),
    \end{align}}
   and,  for  $i  \geq 1$,
   { \allowdisplaybreaks
    \begin{align}
   e_{2i+1}(t)& \equiv 0  \text{ (mod } d),
    \end{align}}
   and
   { \allowdisplaybreaks
    \begin{align}
   e_{2i+2}(t) & \geq  \frac{2  \pi }{ d_{2i+1}^{2}}  \max
    \left  \{  \displaystyle{ \kappa_{n(d_{2i+1})}},  \,  \,
    \frac{ \displaystyle{2 \alpha_{n(d_{2i+1})}}}{ C_{1}}  \right  \}
    \text{ infinitely often}.
    \end{align}}
   Note that,  for $i  \geq 1$,  $d_{2i+1}(t)  \equiv j  \text{ (mod } d)$, so that
   $ \displaystyle{n(d_{2i+1})}$ is well-defined. Note also that $S_{0}$
is an uncountable set.

    \begin{lemma} \label{L:sk}
   For $t  \in S_{0}$,  we have
    { \allowdisplaybreaks
     \begin{align} \label{E:tdif}
     \left |t - \frac{ c_{2i+1}(t)}{d_{2i+1}(t)}  \right |
    &< \frac{1}{2  \pi  \kappa_{n(d_{2i+1})}}
     \end{align}}
    and
    { \allowdisplaybreaks
     \begin{align} \label{E:aldif}
     \left |t - \frac{ c_{2i+1}(t)}{d_{2i+1}(t)}  \right |
    &< \frac{C_{1}}{4  \pi  \alpha_{n(d_{2i+1})}}
     \end{align}}
    for infinitely many $i$.
     \end{lemma}

     \begin{proof}Let $i$ be one of the infinitely many integers for which
     \[
    e_{2i+2}(t)  \geq  \frac{2  \pi }{ d_{2i+1}^{2}}
     \max   \left  \{  \displaystyle{ \kappa_{n(d_{2i+1})}},  \,  \,
     \frac{ \displaystyle{2 \alpha_{n(d_{2i+1})}}}{ C_{1}}  \right  \}
     \]
    and let  $t_{2i+2} = [e_{2i+2}(t), e_{2i+3}(t),  \cdots]$ denote the $2i+2$-th tail of
    the continued fraction expansion for $t$. Then
    { \allowdisplaybreaks
     \begin{align*}
     \left |t - \frac{ c_{2i+1}(t)}{d_{2i+1}(t)}  \right |
    & = \left | \frac{t_{2i+2}c_{2i+1}(t)+c_{2i}(t)}{t_{2i+2}d_{2i+1}(t)+d_{2i}(t)}
    - \frac{ c_{2i+1}(t)}{d_{2i+1}(t)}  \right | \\
&=  \frac{1}{d_{2i+1}(t)(t_{2i+2}d_{2i+1}(t)+d_{2i}(t))} \\
    &< \frac{1}{d_{2i+1}(t)^{2}e_{2i+2}(t)} \leq  \frac{1}{2  \pi  \kappa_{n(d_{2i+1})}}.
     \end{align*}}
     \eqref{E:aldif} follows similarly.
     \end{proof}

     \emph{Proof of Theorem  \ref{T:t3}:} \,  \, Let $t  \in S_{0}$,
    let $y =  \exp(2  \pi i t)$,  and for $i  \geq 1$,  set
    $x_{i} = \exp(2  \pi i c_{2i+1}(t)/d_{2i+1}(t))= \exp(2  \pi i c_{2i+1}/d_{2i+1})$,  so that $x_{i}$ is a
    primitive $d_{2i+1}$-th root of unity,  where $d_{2i+1}  \equiv j$ (mod $d$).
    We use,  in turn,  \eqref{E:qdif},  the fact that chord length is shorter than arc length,
    and  \eqref{E:tdif},  to get that,  for infinitely many $i$,
    { \allowdisplaybreaks
     \begin{align} \label{E:nqm}
    |Q_{n(d_{2i+1})}(y)-Q_{n(d_{2i+1})}(x_{i})| & \leq  \kappa_{n(d_{2i+1})}|y-x_{i}| \\
&< 2  \pi   \kappa_{n(d_{2i+1})}
     \left|t- \frac{c_{2i+1}(t)}{d_{2i+1}(t)} \right| < 1.  \notag
     \end{align}}
    Similarly,
    { \allowdisplaybreaks
     \begin{align} \label{E:nq}
    |Q_{(n(d_{2i+1})-1)}(y)-Q_{(n(d_{2i+1})-1)}(x_{i})|
    & \leq  \kappa_{(n(d_{2i+1}) -1)}|y-x_{i}| \\
&< \frac{  \kappa_{(n(d_{2i+1})-1)}}
    {  \kappa_{n(d_{2i+1})}}< 1.  \notag
     \end{align}}
 The last inequality follows since the sequence $ \{  \kappa_{n}  \}$
 is strictly increasing. Apply the triangle inequality to  \eqref{E:nqm} and
  \eqref{E:nq} and use  \eqref{E:qn} to get that { \allowdisplaybreaks
  \begin{align} \label{E:qnb}
 |Q_{n(d_{2i+1})}(y)| < 1 + C_{2}
  \end{align} }
 and
 { \allowdisplaybreaks
  \begin{align} \label{E:qnb1}
 |Q_{n(d_{2i+1})-1}(y)| < 1 + C_{2}
  \end{align} }
 Similarly,  use  \eqref{E:fdif},  the fact that chord length is shorter than arc length and
  \eqref{E:aldif} to get that,  for the same infinite set of $i$'s,
 { \allowdisplaybreaks
  \begin{align} \label{E:fqm}
 || \chi_{n(d_{2i+1})}(y)|-| \chi_{n(d_{2i+1})}(x_{i})||
 & \leq| \chi_{n(d_{2i+1})}(y)- \chi_{n(d_{2i+1})}(x_{i})|  \\
 & \leq  \alpha_{n(d_{2i+1})}|y-x_{i}| \notag \\
 &< 2  \pi   \alpha_{n(d_{2i+1})} \left|t- \frac{c_{2i+1}(t)}{d_{2i+1}(t)} \right|  \notag \\
 &<  \frac{C_{1}}{2}.  \notag
  \end{align}}
 Use the reverse triangle inequality and  \eqref{E:fnm}to get that
 { \allowdisplaybreaks
  \begin{align} \label{E:fnb}
 | \chi_{n(d_{2i+1})}(y)| > | \chi_{n(d_{2i+1})}(x_{i})|- \frac{C_{1}}{2}  \geq  \frac{C_{1}}{2}.
  \end{align} }
 Use  \eqref{E:qnb},   \eqref{E:qnb1} and  \eqref{E:fnb} to get that
 { \allowdisplaybreaks
  \begin{align*}
  \left| \frac{P_{n(d_{2i+1})}(y)}{Q_{n(d_{2i+1})}(y)}-
  \frac{P_{n(d_{2i+1})-1}(y)}{Q_{n(d_{2i+1})-1}(y)} \right|
 &= \frac{| \chi_{n(d_{2i+1})}(y)|}{|Q_{n(d_{2i+1})}
 (y)Q_{n(d_{2i+1})-1}(y)|} \\
 &> \frac{C_{1}}{2(1+C_{2})^{2}}.
  \phantom{a}
  \end{align*}}
 Since this holds for the infinite set of integers $ \{n(d_{2i+1}) \}_{i=1}^{ \infty}$,
 it follows that $G(y)$ does not converge to a finite value.
 Since $S_{0}$ is an uncountable set,  this proves the theorem.

  \begin{flushright}
 $ \Box$
  \end{flushright}
 Before proving a corollary to this theorem we need the following proposition.
  \begin{proposition}
 Let $q$ be a primitive $2t+1$-th root of unity. Then
 { \allowdisplaybreaks
  \begin{align} \label{E:qprod}
  \prod_{i=1}^{2t}(1+q^{i}) = 1.
  \end{align}}
  \end{proposition}
  \begin{proof}
 Suppose $q =  \exp(2  \pi i a/(2t+1))$,  for some $a  \in  \mathbb{N}$, where $(a, 2t+1)=1$.
 Then
 { \allowdisplaybreaks
  \begin{align*}
  \prod_{i=1}^{2t}(x-q^{i}) =  \frac{x^{2t+1}-1}{x-1}.
  \end{align*}}
 Let $x = -1$ to get the result.
  \end{proof}
 We now prove the following corollary to Theorem  \ref{T:t3}.
  \begin{corollary}For each of the continued fractions $K(q)$, $S_{1}(q)$,  $S_{2}(q)$ and
 $S_{3}(q)$,  there exists an uncountable set of points on the unit circle at which the
 continued fraction does not converge to a finite value.
  \end{corollary}
  \begin{proof}We use some of the information contained in Table  \ref{Ta:t5}.
 { \allowdisplaybreaks
  \begin{table}[ht]
  \begin{center}
  \begin{tabular}{| c | c | c | c | c |}   \hline $G(q)$    & $K(q)$
  & $S_{1}(q)$     &$S_{2}(q)$      &$S_{3}(q)$  \\
    \hline          &      &        &        &   \\
   $(j, d)$           & $(1, 5)$&$(1, 4)$        &$(1, 8)$      & $(1, 6)$  \\
         &      &        &        &   \\
   $n(m)$           & $m-1$  &$2m-1$        &$2m-1$      & $m-1$  \\
       &      &       &         &   \\
       $Q_{n(m)-1}$       & $0$          & $1$        &$1$              & $0$  \\
                 &      &       &         &   \\
   $Q_{n(m)}$       & $q^{(m-1)/5}$& $  (-1)^{(m-1)/4} \,  q^{(m^{2}-1)/8}$ &$q^{(m-1)/2}$ & $q^{(m-1)/3} $  \\
    \hline
    \end{tabular}
    \phantom{asdf} \\
       \caption{} \label{Ta:t5}
       \end{center}
       \end{table}}

      In each case, $q$ is  a primitive $m$-th root of unity and
      $m  \equiv j  \mod{d}$,  where $(j, d)$ is the pair of integers in the
statement of Theorem  \ref{T:t3}.
      The values in the table come from the papers of  Schur  \cite{S17} and
      Zhang  \cite{Z91}.

 For $K(q)$,  take $(j, d)= (1, 5)$ and for $m  \equiv 1$ (mod 5), set $n(m)= m-1$.
      From  \eqref{rreq} we can take $C_{1}=1$ and from Table  \ref{Ta:t5} we can also take
      $C_{2}=1$. This proves the result for $K(q)$.

      For $S_{1}(q)$ take $(j, d)= (1, 4)$ and for $m  \equiv  1$ (mod 4), $m=2 t+1$,
      for some positive integer $t$,  set $n(m)= 2 m-1$.
      Let $x_{m}$ be a primitive $m$-th root of unity. Then,
by  \eqref{E:qprod} and \eqref{z1},
{ \allowdisplaybreaks
 \begin{align*}
| \chi_{n(m)}(x_{m})|= \prod_{i=1}^{2t}|(1+x_{m}^{i})| = 1.
 \end{align*}}
Once again  \eqref{E:fnm} is satisfied with $C_{1}=1$. From  Table
 \ref{Ta:t5} we can once again take $C_{2}=1$. This proves the
result for $S_{1}(q)$.

 For $S_{2}(q)$ take $(j, d)= (1, 8)$ and for $m  \equiv  1$ (mod 8),
 $m=2 t+1$,  for some positive integer
 $t$,   let $n(m)= 2 m-1$. Let $x_{m}$ be a primitive $m$-th root of unity.
 Then, from \eqref{z2},
 { \allowdisplaybreaks
  \begin{align*}
  \left| \chi_{n(m)}(x_{m}) \right|&= \prod_{i=1}^{2t+1} \left|(1+x_{m}^{2i-1}) \right|  \\
 &= \prod_{i=1}^{t} \left|(1+x_{m}^{2i-1}) \right|
  \left|(1+x^{2t+1}) \right| \prod_{i=t+2}^{2t+1} \left|(1+x_{m}^{2i-1}) \right| \\
 &=2 \prod_{i=1}^{t} \left|(1+x_{m}^{2i-1}) \right| \prod_{i=1}^{t} \left|(1+x_{m}^{2i}) \right| \\
 &=2 \prod_{i=1}^{2t} \left|(1+x_{m}^{i}) \right|=2.
  \end{align*}
 }
 The last equality uses  \eqref{E:qprod}. In this case,   \eqref{E:fnm} is satisfied with $C_{1}=2$.
 From  Table  \ref{Ta:t5} we can once again take $C_{2}=1$. This proves the result for $S_{2}(q)$.

 Finally, for $S_{3}(q)$ take $(j, d)= (1, 6)$ and for $m  \equiv  1$ (mod  6),
 $m=2 t+1$,  for some positive integer $t$,   let  $n(m) =m-1$.
 Let $x_{m}$ be a primitive $m$-th root of unity. By   \eqref{E:qprod},
 { \allowdisplaybreaks
  \begin{align*}
 | \chi_{n(m)}(x_{m})|= \prod_{i=1}^{2t}|(1+x_{m}^{i})| = 1.
  \end{align*}}
 Once again  \eqref{E:fnm} is satisfied with $C_{1}=1$. From Table  \ref{Ta:t5},
 it is clear that
 we can take $C_{2}=1$. This proves the result for $S_{3}(q)$.
  \end{proof}

  \section{Concluding Remarks}In proving the existence of an uncountable set of points
on the unit circle
 at which a $q$-continued fraction $G(q)$ does not converge to finite values,
  our methods
 rely on knowing the behavior of
 the continued fraction at roots of unity and,  if $q$ is a primitive$m$-th root of unity,
 on knowing
 something about the values of$| \chi_{n(m)}(q)|$,  $| Q_{n(m)}(q)|$ and $| Q_{n(m)-1}(q)|$
 (see the statement of Theorem  \ref{T:t3}). It would be interesting to have a criterion
based on the
 $a_{n}(q)$ and the $b_{n}(q)$ which would indicate whether  \eqref{E:fnm} and
\eqref{E:qn}were satisfied,
 as this would automatically give information about the convergence
behavior of the continued fraction
 on the unit circle away from roots of unity.

 Our Theorem  \ref{T:t3} can show the existence of a set of
 measure zero on the unit circle at which a $q$-continued fraction does not
converge to finite values.
 However, it seems likely that the particular continued fractions that we have looked at
 diverge almost everywhere, in both the classical and general senses, on the unit circle.
As yet,  we do not see how to prove this almost
 everywhere divergence.

The most famous $q$-continued fraction after the Rogers-Ramanujan continued
 fraction is the G\"{o}llnitz-Gordon continued fraction,
{ \allowdisplaybreaks
 \begin{align} \label{ggcf}
GG(q) :=1+q +  \frac{q^{2}}{1+q^{3}} \+  \frac{q^{4}}{1+q^{5}} \+
 \frac{q^{6}}{1+q^{7}} \+ \,  \cds.
 \end{align}}
  This continued fraction tends to the same limit as $S_{2}(q)$,
 for each $q$ inside the unit circle,  but the behavior at roots of unity is slightly different.
 Based on computer investigations of the behavior of $GG(q)$ at roots of unity,  it would seem that the
 following is true. Both $GG(q)$ and $S_{2}(q)$ agree at $m$-th roots of unity
(either both converge to the same limit or both diverge),
 except if $m  \equiv 2 \text{(mod 4)}$,  in which case $GG(q)$ diverges and $S_{2}(q)$ converges.
 Also,  computer evidence also seems to suggest that $GG(q)$satisfies the conditions on the
 $a_{i}(q)$ and the $Q_{i}(q)$ required by Theorem   \ref{T:t3},   implying that there is an uncountable set
 of points on the unit circle at which  the G\"{o}llnitz-Gordon continued fraction does
not converge to finite values.
 However,  these facts have not yet been proved. We hope to do this in a later paper and thereby
 show that the G\"{o}llnitz-Gordon continued fraction does indeed diverge at uncountably many
 points on the unit circle.

In  \cite{BML02} we extend our results in  \cite{BML01} on the divergence
 in the  \emph{general}  sense (see,  for example,   \cite{J86} and  \cite{LW92})
of the Rogers-Ramanujan
 continued fraction on the unit circle. We show,  for each $q$-continued fraction
 in a certain class of $q$-continued fractions (a class which  includes the Rogers-Ramanujan continued
  fraction and the three Ramanujan-Selberg continued fractions),  that there exists an uncountable
  set of points on the unit circle at which the continued fraction diverges in the general sense.

{ \allowdisplaybreaks
 }

 \end{document}